\documentclass[10pt,draft,reqno]{amsart}

\makeatletter
     \def\section{\@startsection{section}{1}%
     \z@{.7\linespacing\@plus\linespacing}{.5\linespacing}%
     {\bfseries
     \centering
     }}
     \def\@secnumfont{\bfseries}
     \makeatother

\newtheorem{theorem}{Theorem}[section]
\newtheorem{lemma}[theorem]{Lemma}
\newtheorem{proposition}[theorem]{Proposition}
\theoremstyle{definition}

\theoremstyle{remark}
\newtheorem{remark}[theorem]{Remark}
\numberwithin{equation}{section}
\begin{document}
\title[CONVERGENCE TO WEIGHTED FRACTIONAL BROWNIAN SHEETS]{CONVERGENCE TO WEIGHTED FRACTIONAL BROWNIAN SHEETS*}
\author{JOHANNA GARZ\'ON }
\thanks{* Research partially supported by CONACYT grant 45684-F}
\address{Department of Mathematics, CINVESTAV, Mexico City, Mexico}
\email{johanna@math.cinvestav.mx}


\begin{abstract}We define weighted fractional Brownian sheets, which are a class of Gaussian random fields with four parameters that include fractional Brownian sheets as special cases, and we give some of their properties. We show that for certain values of the parameters the weighted fractional Brownian sheets are obtained as limits in law of occupation time fluctuations of a stochastic particle model. In contrast with some known approximations of fractional Brownian sheets which use a kernel in a Volterra type integral representation of fractional Brownian motion with respect to ordinary Brownian motion, our approximation does not make use of a kernel.
\end{abstract}

\maketitle
\section{Introduction}

\allowdisplaybreaks

Fractional Brownian sheets have been studied by several authors for their mathematical interest and their applications. One of the first papers on the subject is \cite{kamont}. Some types of approximations of fractional Brownian sheets have been obtained recently (e.g. \cite{bardina}, \cite{bardinaflorit}, \cite{coutin}, \cite{kuhn}, \cite{nzi}, \cite{tudor}). In this paper we give a new type of approximation for certain values of the parameters by means of occupation time fluctuations of a stochastic particle model. The  limits that are obtained in this way   are a more general class of Gaussian random fields.

We consider centered Gaussian random fields $W=\left(W_{s,t}\right)_{s,t \geq 0}$ with parameters $(a_i, b_i)$, $i=1, 2$, whose covariance is given by
\begin{equation}
\label{eq1}
K_W((s, t), (s', t'))= E\left(W_{s,t}W_{s',t'}\right)= C^{(1)}(s, s')C^{(2)}(t,t'), 
\end{equation}
where each  $C^{(i)}$ is of the form
\begin{equation}
\label{equ2}
C^{(i)}(u,v)= \int_{0}^{u\wedge v} r^{a_i}[(u-r)^{b_i} + (v-r)^{b_i}]dr,   \  \  \  \  i=1,2,
\end{equation}
with the following ranges for the parameters:
\begin{equation}
\label{equ2'}
a_i> -1,\  -1<b_i \leq 1, \ \left|b_i\right|\leq 1+a_i.
\end{equation}
$C^{(i)}$ is the covariance of weighted fractional Brownian motion with parameters $(a_i, b_i)$. Weighted fractional Brownian motions were introduced in \cite{bgt2003}. We call $W$ a weighted fractional Brownian sheet with parameters $(a_i, b_i)$, $i=1,2.$  In the case $a_1=a_2=0$ (the weight functions are 1) $W$ is a fractional Brownian sheet  with parameters $(\frac{1}{2}(1+b_1), \frac{1}{2}(1+b_2))$. The case $a_1=a_2=b_1=b_2=0$ corresponds to the ordinary Brownian sheet. If $b_1=b_2=0,$ and at least one $a_i$ is not $0$, then $W$ is a time-inhomogeneous  Brownian sheet.

Due to the covariance structure (\ref{eq1}), (\ref{equ2}), many properties of $W$ are consequences of those of weighted fractional Brownian motion. We will prove an approximation in law  of $W$ for $a_i $ and $b_i$ of the form
$a_i=-\gamma_i/\alpha_i, b_i=1-1/\alpha_i$, with $0\leq \gamma_i < 1$ and $1<\alpha_i \leq 2$; hence the approximation is restricted to values of $a_i$ and $b_i$ such that $-1 < a_i \leq 0$ and $0 < b_i < 1+ a_i$, $i=1, 2.$ 

The approximations of fractional Brownian sheets in \cite{bardina},  \cite{tudor} are based on a Poisson random measure on $\mathbb{R}_{+}\times\mathbb{R}_+$ and a kernel which appears in a Volterra type integral representation of fractional Brownian motion with respect to ordinary Brownian motion. The approximation in \cite{bardinaflorit}, analogous to the functional invariance theorem, also uses the kernel. Our approach does not use a kernel. We also use a Poisson random measure, but on $\mathbb{R}\times\mathbb{R}$ instead of $\mathbb{R}_{+}\times\mathbb{R}_+$ and in a different way from \cite{bardina},  \cite{tudor}. Some of the other approximations cited above are motived by simulation of fractional Brownian sheets. Our approximation is not intended for simulation, but rather to show that weighted fractional Brownian sheets emerge in a natural way from a simple particle model.

In section 2 we give the properties of $W$, in particular long-range dependence. In section 3 we describe the particle system and we prove convergence to $W$ of rescaled ocupation time fluctuations of the system  for the above mentioned values of the parameters.

\section{Properties}

We consider $\mathbb{R}^2_+$  with the following partial order: for
 $z=(s,t)$ and $ z'=(s', t')$, \ 
$ z \preceq z'$  iff $  s\leq s'$ and $ t \leq t',$
$ z \prec z'$ iff $ s< s'$ and $t < t'$, and if $z \prec z'$ we denote by $(z, z']$ the rectangle $(s,s']\times (t,t']$. We refer to elements of $\mathbb{R}_+^2$ as ``times"  for simplicity of exposition.

If $X=\left(X_z\right)_{ z\in \mathbb{R}_+^2}$ is a two-time stochastic process, the increment of $X$ over the rectangle $(z,z']$ with $z=(s,t), z'=(s',t')$ is defined by
 $$X((z,z'])\equiv\Delta_{s,t}X(s',t'):= X_{(s',t')}-X_{(s,t')}-X_{(s',t)}+X_{(s,t)}.$$

 We denote the covariance of the increments of the process $X$ over the rectangles $((s,t),(s',t')]$, $((p, r),(p', r')]$
 by
 $$K_X\left((s,t), (s',t'); (p, r), (p', r') \right)=Cov\left(\Delta_{s,t}X(s',t'),\Delta_{p,r}X(p',r')\right).$$
The covariance of $W$ over rectangles is given by
\begin{align}
&K_W((s,t), (s',t'); (p, r), (p', r'))\notag\\
&=(C^{(1)}(s',p') - C^{(1)}(s,p') - C^{(1)}(s',p) + C^{(1)}(s,p))\notag\\
&\times (C^{(2)}(t',r') - C^{(2)}(t,r') - C^{(2)}(t',r) +C^{(2)}(t,r))\notag\\
&= Cov(Y_{s'}^{(1)}-Y_{s}^{(1)}, Y_{p'}^{(1)}-Y_{p}^{(1)})Cov(Y_{t'}^{(2)}-Y_{t}^{(2)}, Y_{r'}^{(2)}-Y_{r}^{(2)}),
\end{align}
where  $Y^{(i)}$ is  weighted fractional Brownian motion with parameters $(a_i, b_i)$, $i=1,2.$ 

The next theorem contains some properties of weighted fractional Brownian sheets.
\begin{theorem}
 The weighted fractional Brownian sheet $W$ with parameters $(a_i, b_i)$, $i=1,2,$ has the following properties:
\begin{enumerate}
	\item Self-similarity:
	\begin{equation}
\label{eqautosimilar}
\left(W_{hs, kt}\right)_{s,t \geq 0}\overset{\text{d}}{=} h^{(1+a_1+b_1)/2}k^{(1+a_2+b_2)/2}(W_{s,t})_{s,t \geq 0} \  \  \text{for each} \  \  h,k > 0,
\end{equation}
where $\overset{\text{d}}{=}$ denotes equality in distribution.
\item $W$ has stationary increments only in the case $a_1=a_2=0$.
\item Covariance of increments: For $(0,0)\preceq (s,t) \prec (s',t') \preceq (p, r) \prec (p', r')$,
\begin{align}
\label{eq2}
&K_W\left((s,t), (s',t'); (p, r), (p', r') \right)\notag\\
&=  \int_{s}^{s'} u^{a_1}[(p'-u)^{b_1} + (p-u)^{b_1}]du   \int_{t}^{t'} v^{a_2}[(r'-v)^{b_2} + (r-v)^{b_2}]dv,
\end{align}
hence
\[K_W\left((s,t), (s',t'); (p, r), (p', r') \right)
\begin{cases}
>0 & \text{if} \  \   b_1b_2 > 0,	\\
=0 & \text{if} \  \   b_1b_2 = 0,\\
<0 & \text{if} \  \   b_1b_2 < 0.
\end{cases}\]
\item The one-time processes $(W_{s,t})_{s\geq 0}$ ($t$ fixed) and $(W_{s, t})_{t\geq 0}$ ($s$ fixed) are weighted fractional Brownian motions (multiplied by constants) with parameters $(a_1, b_1)$ and $(a_2, b_2)$, respectively.
\item
\begin{equation}
\label{eqesperanza}
E(\left(\Delta_{s,t}W_{s',t'}\right)^2)= 4\int_s^{s'}u^{a_1}(s'-u)^{b_1}du\int_t^{t'}v^{a_2}(t'-v)^{b_2}dv.
\end{equation}
\item 
\begin{equation}
\lim_{\varepsilon, \delta \to 0}\varepsilon^{-b_1-1}\delta^{-b_2-1}E(\left(\Delta_{s,t}W_{s+\varepsilon, t+\delta}\right)^2)= \frac{4}{(1+b_1)(1+b_2)}s^{a_1}t^{a_2},
\end{equation}
\begin{align}
&\lim_{T, S \to \infty}S^{-(1+a_1+b_1)}T^{-(1+a_2+b_2)}E(\left(\Delta_{s,t}W_{s+S, t+T}\right)^2)\notag\\
&=4 \int_0^1u^{a_1}(1-u)^{b_1}du \int_0^1v^{a_2}(1-v)^{b_2}dv,
\end{align}
hence $W$ has asymptotically stationary increments for long increments in $\mathbb{R}_+^2$, but not for short ones (if $a_1,a_2\neq 0$).
\item The finite-dimensional distributions of the process 
$$(S^{-a_1/2}T^{-a_2/2}\Delta_{S,T}W_{s+S,t+T})_{s,t\geq 0}$$
 converge as $T, S \to \infty$ to those of fractional Brownian sheet with parameters $(\frac{1}{2}(1+b_1), \frac{1}{2}(1+b_2))$ multiplied by $2/[(1+b_1)(1+b_2)]^{1/2}$.
\item Long-range dependence: for $(s,t)\prec (s',t'),$ $(p,u)\prec(p',u')$
\begin{align}
&\lim_{\tau, \kappa \to \infty}\tau^{1-b_1}\kappa^{1-b_2}K_W\left((s,t), (s',t'); (p+\tau , u+ \kappa), (p'+\tau, u'+\kappa) \right)\notag\\
&=\frac{b_1b_2}{(1+a_1)(1+a_2)}(p'-p)((s')^{1+a_1}-s^{1+a_1})(u'-u)((t')^{1+a_2}- t^{1+a_2}).
\end{align}
\item For $\theta>0$, we define the one-time process $\left(Z_t\right)_{t\geq 0 }=\left(W_{t, \theta t}\right)_{t\geq 0 }$, i.e., the sheet restricted to a ray through the origen. (Note that $Z$ is not a weighted fractional Brownian motion.) Then for $0\leq b_i < 1$, $i=1,2$, and not both $b_1, b_2$ equal to $0$,  this process has the long-range dependence property
\begin{align}
&\lim_{\tau \to \infty}\tau^{1-(b_1+b_2)}Cov(Z_v - Z_u, Z_{t+\tau}-Z_{s+\tau})\notag\\
&=\frac{\theta^{1+a_2+b_2}(b_1+b_2)}{(1+a_1)(1+a_2)}(v^{2+a_1+a_2}-u^{2+a_1+a_2})(t-s), \  \ \ u<v, s<t.	
\end{align}
\end{enumerate}
\end{theorem}
\begin{proof}
 Except for part (9), the proofs follow directly from the form of $K_W$ given by (\ref{eq1}), (\ref{equ2}) and  properties of weighted fractional Brownian motion  \cite{bgt2003}. We give an outline of the proof of part (9).
 
 We have, for $u<v$, $s<t$,
 \begin{align}
 \label{e30}
&Cov(Z_v-Z_u, Z_{t+\tau}-Z_{s+\tau})\notag\\
&=\theta^{1+a_2+b_2}[C^{(1)}(v,t+\tau)C^{(2)}(v,t+\tau) - C^{(1)}(v,s+\tau)C^{(2)}(v,s+\tau)\notag\\
&\hspace{1.7cm} -C^{(1)}(u,t+\tau)C^{(2)}(u,t+\tau) + C^{(2)}(u,s+\tau)C^{(2)}(u,s+\tau)],
\end{align}
 \begin{align}
  \label{e31}
&C^{(1)}(v,t+\tau)C^{(2)}(v,t+\tau) - C^{(1)}(v,s+\tau)C^{(2)}(v,s+\tau)\notag\\
&=[C^{(1)}(v,t+\tau)- C^{(1)}(v,s+\tau)]C^{(2)}(v,t+\tau)\notag\\
&+[C^{(2)}(v,t+\tau)- C^{(2)}(v,s+\tau)]C^{(1)}(v,s+\tau)\notag\\
&=\int_0^vr^{a_1}[(t-r+\tau)^{b_1}-(s-r+\tau)^{b_1}]dr\int_0^vr^{a_2}[(t-r+\tau)^{b_2}+(v-r)^{b_2}]dr\notag\\
&+\int_0^vr^{a_2}[(t-r+\tau)^{b_2}-(s-r+\tau)^{b_2}]dr\int_0^vr^{a_1}[(s-r+\tau)^{b_1}+(v-r)^{b_1}]dr,
\end{align}
and similarly for the last two terms. The result follows from (\ref{e30}), (\ref{e31}) and the limits
\begin{equation*}
	\lim_{\tau \to \infty}\tau^{1-b}[(t_2+\tau)^b-(t_1+\tau)^b]=b(t_2-t_1)
\end{equation*}
and
\begin{equation*}
		\lim_{\tau \to \infty}\tau^{-b}\int_0^vr^{a}[(t+\tau)^{b}+(v-r)^{b}]dr=\frac{v^{1+a}}{1+a}.
\end{equation*}
\end{proof}

\begin{remark}
There are three different long-range dependence regimes in property (9), and they are independent of $a_1, a_2$. The covariance of increments of $Z$ has a power decay for $b_1+b_2<1$, a power growth for $b_1+b_2>1$, and a non-trivial limit for $b_1+b_2=1$. We do not know if this property has been noted before for fractional Brownian sheets. It is worthwhile to observe that the non-Gaussian process $(Y_t^{(1)}Y_{\theta t}^{(2)})_{t\geq 0}$, where $Y^{(i)}$ are independent weighted fractional Brownian motions with parameters $(a_i, b_i),$ $i=1,2$, has the same long-range dependence behavior. 
\end{remark}

In \cite{bgt2003} it is  shown that $A_{s,t}= \int_s^t u^a(t-u)^b du$,  $0 \leq s < t$, has the following bounds:
If $a\geq 0,$ $s,t \leq T$ for any $T>0$ and constant $M=M(T)$, and also if $a< 0,$  $s,t\geq \varepsilon $ for any $\varepsilon>0$ and constant $M=M(\varepsilon)$, 
	$$A_{s,t} \leq M\left|t-s\right|^{1+b}.$$
 If $a<0$, $1+a+b>0$, $s, t\geq 0$,
	$$A_{s,t} \leq M\left|t-s\right|^{1+a+b}.$$
Then it follows from (\ref{eqesperanza}) that for $0 < \varepsilon \leq s < s' <T,$ $0 < \varepsilon \leq t < t' <T$ and $i=1,2,$
\begin{equation}
\label{eq26}
E(\left(\Delta_{s,t}W_{s',t'}\right)^2)\leq M\left(s'-s\right)^{\delta_1}\left(t'-t\right)^{\delta_2},	
\end{equation}
where
\begin{equation}
\label{eqdeltas}
\delta_i=
\Biggl\{\begin{array}{ll}
1+a_i+b_i & \text{\ \ \ if\ \ \ } a_i< 0 \text{\ \ \ \ and\ \ \ } 1+a_i+b_i>0,\\
1+b_i & \text{\ \ \ otherwise.}
\end{array}	\Biggr.
\end{equation}

The next lemma  allows us to prove the continuity of $W$.  	

\begin{lemma}
\label{lemacontinuidad}
\cite{ayache}, \cite{feyel} Let $X=(X_{s,t})_{s,t\geq 0 }$ be a two-time stochastic process on a probability space $(\Omega, \mathfrak{F}, P)$ which is null almost surely on the axes and such that there exist $p>0$, $a,b\in (1/p, \infty)$, such that
$$(E(\left|\Delta_{s,t} X_{s+h, t+k}\right|^p))^{1/p}\leq M|h|^a|k|^b.$$
Then $X$ has a modification $\Tilde{X}$ with continuous trajectories. Also, the trajectories of  $\Tilde{X}$ are H\"older with exponents $(a', b')$, for $a' \in (0, a- 1/p)$, $b' \in (0, b- 1/p)$, that is, for any $\omega \in \Omega$ exists $M_{\omega}>0$ such that for any $s,s',t,t',$
$$|\Delta_{s,t}\Tilde{X}_{s',t'}(\omega)|\leq M_{\omega}(s'-s)^{a'}(t'-t)^{b'}, \ \ \ \ \  s<s',\  t<t'.$$
\end{lemma}

\begin{proposition}
 The weighted fractional Brownian sheet $(W_{s,t})_{s,t\geq 0 }$ has a modification $(\Tilde{W}_{s,t})_{s,t\geq 0}$ with continuous trajectories. Also, the trajectories of  $\Tilde{W}$ are H\"older with exponents $(x, y)$ for any $x \in (0, \frac{1}{2}\delta_1), y \in (0, \frac{1}{2}\delta_2)$,
where $\delta_i$ are as in (\ref{eqdeltas}).
\end{proposition}

\begin{proof}
 From the moments of the normal distribution  and  equations (\ref{eqesperanza}) and (\ref{eq26}) we have
\begin{align*}
&(E(|\Delta_{s,t}W_{s+h,t+k}|^r))^{1/r}\\
&=C\biggl(\int_s^{s+h}u^{a_1}(s+h-u)^{b_1}du\int_t^{t+k}v^{a_2}(t+k-v)^{b_2}dv\biggr)^{1/2}\\
&\leq Mh^{\delta_1/2}k^{\delta_2/2},
\end{align*}
with some constants $C$ and $M$.  Taking $r> \max\left\{2/\delta_1, 2/\delta_2\right\}$ we have the conditions of Lemma \ref{lemacontinuidad}, and the result follows. 
\end{proof}

 \section{Approximation }

The random field $W$, for some values of the parameters $a_i, b_i$, arises as a limit in distribution of occupation time fluctuations of a system of particles of two types that move as pairs in $\mathbb{R}\times \mathbb{R}$ according to independent stable L\'evy  processes. The  system is described as follows. Given a  Poisson random measure on $\mathbb{R}\times \mathbb{R}$ with intensity measure $\mu$, $N_{0,0}= \text{Pois}(\mu)$, from each point $(x_1, x_2)$ of $N_{0,0}$ come out two independent L\'evy processes, from $x_1$ comes out $\xi^{x_1} $, symmetric $\alpha_1$-stable,  and from $x_2$ comes out $\zeta^{x_2}$, symmetric $\alpha_2$-stable $(0<\alpha_i \leq 2,\   i=1,2)$. Let $N=(N_{u,v})_{ u, v \geq 0}$ denote  random measure process on $\mathbb{R}\times \mathbb{R}$ such that  $N_{u,v}$  represents the configuration of particles at time $(u,v)$,
\begin{equation}
\label{eqdefN}
N_{u,v}= \sum_{(x_1, x_2)\in N_{0,0}}\delta_{(\xi^{x_1}_u, \zeta^{x_2}_v)}=\sum_{(x_1, x_2)\in N_{0,0}}\delta_{\xi^{x_1}_u}\otimes\delta_{\zeta^{x_2}_v}.
\end{equation}
For $\varphi, \psi \in L^1(\mathbb{R})$ ($\varphi, \psi \neq 0$) fixed, we write
\begin{equation}
\label{eqdefN2}
\langle N_{u,v}, \varphi\otimes \psi\rangle= \sum_{(x_1, x_2)\in N_{0,0}}\langle \delta_{\xi^{x_1}_u}\otimes\delta_{\zeta^{x_2}_v},\varphi\otimes \psi \rangle = \sum_{(x_1, x_2)\in N_{0,0}}\varphi(\xi^{x_1}_u)\psi(\zeta^{x_2}_v).
\end{equation}
We define the occupation time process of $N$ by
\begin{equation}
\label{eqdefL}
\left\langle L_{s,t}, \varphi\otimes \psi\right\rangle= \int_0^s\int_0^t \left\langle N_{u,v}, \varphi\otimes \psi\right\rangle dv du, \  \  \ \ s,t\geq 0,
\end{equation}
and the rescaled occupation time fluctuation process by
\begin{equation}
\label{eqdefXT}
X_T(s,t)= \frac{1}{F_T} \left(\left\langle L_{Ts, Tt},\varphi\otimes \psi\right\rangle - E(\left\langle L_{Ts, Tt}, \varphi\otimes \psi\right\rangle\right)), \  \  \  \  s,t\geq 0,
\end{equation}
where $T$ is the  time scaling and $F_T$ is a norming. We choose the intensity measure $\mu$ for the Poisson initial particle configuration as
$$\mu(dx_1, dx_2)=\mu_1\otimes \mu_2(dx_1, dx_2)=\mu_1(dx_1)\mu_2(dx_2),$$
with
\begin{equation}
\label{eqmu}
\mu_i(dx_i)=dx_i/\left|x_i\right|^{\gamma_i}, \  0 \leq \gamma_i < 1, \  i=1,2.
\end{equation}

The homogeneous case corresponds to  $\gamma_1=\gamma_2= 0$ and it gives rise to the usual fractional Brownian sheet.  We will show that for 
\begin{equation}
\label{eqdefFT}
F_T = F_T^{(1)}F_T^{(2)}\  \text{with} \ F^{(i)}_T= T^{1-(1+\gamma_i)/2\alpha_i}, \ 0\leq \gamma_i < 1 < \alpha_i, \  i=1, 2,	
\end{equation} 
the finite-dimensional distributions of the process $X_T $ converge in law as $T \to \infty$ to those of weighted fractional Brownian sheet with parameters $a_i=-\gamma_i/\alpha_i$, $b_i= 1 - 1/\alpha_i$, $i=1,2$. In the case $a_1=a_2=0$ we will also prove tightness.

\begin{theorem}
\label{teo1}
If $X_T$ is the process defined in (\ref{eqdefXT}),  $0\leq\gamma_i < 1 < \alpha_i$, $i=1,2$, with $F_T$ defined by (\ref{eqdefFT}), then the finite-dimensional distributions of $X_T$ converge  as $T\to \infty$ to the finite-dimensional distributions of $DW$, where $W$ is weighted fractional Brownian sheet with parameters $a_1=-\gamma_1/\alpha_1, b_1= 1- 1/\alpha_1$, $a_2=-\gamma_2/\alpha_2$, $b_2= 1- 1/\alpha_2$, and $D$ is the constant 
\begin{equation}
\label{defD}
D=\int_{\mathbb{R}}\varphi(x)dx\int_{\mathbb{R}}\psi(x)dx\biggl(\prod_{i=1}^{2}\frac{1}{1-1/\alpha_i}p_1^{\alpha_i}(0)\biggl(\int_{\mathbb{R}}\frac{p_1^{\alpha_i}(x)}{|x|^{\gamma_i}}dx\biggr)\biggr)^{1/2},
\end{equation}
where $p_t^{\alpha}(x)$ is the density of the symmetric $\alpha$-stable L\'evy process, which is given by
$$p_t^{\alpha}(x)=\frac{1}{2\pi}\int_{\mathbb{R}}\exp\left\{-\left(ixy + t|y|^{\alpha}\right)\right\}dy.$$
\end{theorem}

\begin{proof}
For each $k \in \mathbb{N}$,  $d_1, \cdots , d_k \in \mathbb{R}$ and $(s_1, t_1), \cdots, (s_k, t_k)\in \mathbb{R}_+^2$, we must show that 
$$\sum_{j=1}^k d_j X^T_{s_j, t_j}\ \ \ \text{converges in law to}\ \ \  \sum_{j=1}^k d_j W_{s_j, t_j} \ \ \ \text{as} \  T\to \infty,$$
which we do by proving convergence of the corresponding characteristic functions. From the fact that $N_{0,0}=\text{Pois}(\mu_1\otimes\mu_2)$, we have for each $\theta \in \mathbb{R}$,
\begin{align}
\label{eq3}
&C_T(\theta):= E\exp\biggl\{i\theta \sum_{j=1}^k d_j X^T_{s_j, t_j}\biggr\}\notag\\ 
&=\exp\biggl\{-\frac{i\theta}{F_T}\sum_{j=1}^k d_j E(\langle L^T_{s_j, t_j},\varphi\otimes \psi\rangle)\biggr\} \notag\\
&\times \exp\biggl\{-\int_{\mathbb{R}\times\mathbb{R}}\biggl[1-E_{(x_1, x_2)}\biggl(\exp\biggl\{\frac{i\theta}{F_T}\sum_{j=1}^k d_j \langle L^T_{s_j, t_j},\varphi\otimes \psi\rangle\biggr\}\biggr)\biggr]\mu_1(dx_1)\mu_2(dx_2)\biggr\},
\end{align}
where $E_{(x_1, x_2)}$ denotes expectation starting with one pair of initial particles in $(x_1, x_2)$, (see e.g. \cite{kallenberg}, mixed Poisson process).

We also need the mean and the covariance of $N$. From the Poisson initial condition, the first and second moments are given by
\begin{align}
 \label{e98}
E\left(\left\langle N_{u,v}, \varphi\otimes \psi \right\rangle\right)&=\int_{\mathbb{R}\times \mathbb{R}}E_{(x_1, x_2)}\left(\left\langle N_{u,v}, \varphi\otimes \psi \right\rangle\right)\mu_1(dx_1)\mu_2(dx_2)\notag \\
&=\int_{\mathbb{R}\times \mathbb{R}}E\left(\varphi(\xi^{x_1}_u)\psi(\zeta^{x_2}_v)\right)\mu_1(dx_1)\mu_2(dx_2)
\end{align}
and
\begin{align*}
&E\left(\left\langle N_{u_1,v_1}, \varphi\otimes \psi \right\rangle\left\langle N_{u_2,v_2}, \varphi\otimes \psi \right\rangle\right)\notag\\
&=\int_{\mathbb{R}\times \mathbb{R}}E_{(x_1, x_2)}\left(\left\langle N_{u_1,v_1}, \varphi\otimes \psi \right\rangle\left\langle N_{u_2,v_2}, \varphi\otimes \psi \right\rangle\right)\mu_1(dx_1)\mu_2(dx_2)\notag \\
&+\int_{\mathbb{R}\times \mathbb{R}}E_{(x_1, x_2)}\left(\left\langle N_{u_1,v_1}, \varphi\otimes \psi \right\rangle\right)\mu_1(dx_1)\mu_2(dx_2)\notag\\
&\times\int_{\mathbb{R}\times \mathbb{R}}E_{(x_1, x_2)}\left(\left\langle N_{u_2,v_2}, \varphi\otimes \psi \right\rangle\right)\mu_1(dx_1)\mu_2(dx_2)\notag\\
&=\int_{\mathbb{R}\times \mathbb{R}}E\left(\varphi(\xi^{x_1}_{u_1})\psi(\zeta^{x_2}_{v_1})\varphi(\xi^{x_1}_{u_2})\psi(\zeta^{x_2}_{v_2})\right)\mu_1(dx_1)\mu_2(dx_2)\notag \\
&+\int_{\mathbb{R}\times \mathbb{R}}E\left(\varphi(\xi^{x_1}_{u_1})\psi(\zeta^{x_2}_{v_1})\right)\mu_1(dx_1)\mu_2(dx_2)\int_{\mathbb{R}\times \mathbb{R}}E\left(\varphi(\xi^{x_1}_{u_2})\psi(\zeta^{x_2}_{v_2})\right)\mu_1(dx_1)\mu_2(dx_2),\notag\\
\end{align*}
hence, by the independence of $\xi$ and $\zeta$, and the Markov property, 
\begin{align}
\label{defcov}
&Cov\left(\left\langle N_{u_1,v_1}, \varphi\otimes \psi \right\rangle,\left\langle N_{u_2,v_2}, \varphi\otimes \psi \right\rangle\right)\notag \\
&=\int_{\mathbb{R}\times \mathbb{R}}E_{(x_1, x_2)}\left(\left\langle N_{u_1,v_1}, \varphi\otimes \psi \right\rangle\left\langle N_{u_2,v_2}, \varphi\otimes \psi \right\rangle\right)\mu_1(dx_1)\mu_2(dx_2)\notag \\
&=\int_{\mathbb{R}}\mathcal{T}^{\alpha_1}_{u_1\wedge u_2}(\varphi\mathcal{T}^{\alpha_1}_{|u_1-u_2|}\varphi)(x_1)\mu_1(dx_1)\int_{\mathbb{R}}\mathcal{T}^{\alpha_2}_{v_1\wedge v_2}(\psi\mathcal{T}^{\alpha_2}_{|v_1-v_2|}\psi)(x_2)\mu_2(dx_2),
\end{align}
where  $\mathcal{T}^{\alpha_i}_t$ denotes the semigroup of the symmetric $\alpha_i$-stable process.

Using an expansion of the characteristic function (see e.g \cite{billins}, p. 297) in the integrand  with respect to $(x_1,x_2)$ in (\ref{eq3}), it is  equal to
\begin{align*}
& 1+\frac{i\theta}{F_T}E_{(x_1, x_2)}\biggl(\sum_{j=1}^k d_j\langle L^T_{s_j, t_j},\varphi\otimes \psi \rangle\biggr)- \frac{\theta^2}{2F_T^2}E_{(x_1, x_2)}\biggl(\sum_{j=1}^k d_j\langle L^T_{s_j, t_j},\varphi\otimes \psi \rangle\biggr)^2
\\
&+\delta_{(x_1, x_2)}^T,
\end{align*}
where 
\begin{equation}
\label{eqdelta}
	|\delta_{(x_1, x_2)}^T|\leq \frac{\theta^3}{F_T^3}E_{(x_1, x_2)}\biggl(\sum_{j=1}^k d_j\langle L^T_{s_j, t_j},\varphi\otimes \psi\rangle\biggr)^3.
\end{equation}
Since 
$$\sum_{j=1}^k d_jE(\langle L^T_{s_j, t_j},\varphi\otimes \psi\rangle)=\int_{\mathbb{R}\times\mathbb{R}}\sum_{j=1}^k d_jE_{(x_1, x_2)}(\langle L^T_{s_j, t_j},\varphi\otimes \psi\rangle)\mu_1(dx_1)\mu_2(dx_2),$$
then (\ref{eq3}) becomes
\begin{align}
\label{eq4}
C_T(\theta)=\exp\biggl\{\biggr.-\int_{\mathbb{R}\times\mathbb{R}}\biggl[\biggr.&\frac{\theta^2}{2F_T^2}E_{(x_1, x_2)}\biggl(\sum_{j=1}^k d_j\langle L^T_{s_j, t_j},\varphi\otimes \psi\rangle\biggr)^2  \notag\\ 
&+ \biggl.\biggl.\delta_{(x_1, x_2)}^T\biggr] \mu_1(dx_1)\mu_2(dx_2)\biggr\}  
\end{align}
and by (\ref{eqdefL}) and a previous calculation,
 \begin{align}
 \label{eq50}
 &\int_{\mathbb{R}\times \mathbb{R}}\frac{1}{F_T^2}E_{(x_1,x_2)}\biggl(\sum_{j=1}^{k}d_j\langle L_{s,t}^T, \varphi\otimes \psi\rangle\biggr)^2\mu_1(dx_1)\mu_2(dx_2)\notag\\
 =&\frac{1}{F_T^2}\sum_{j=1}^{k}d_j\sum_{j'=1}^{k}d_{j'}\int_{\mathbb{R}\times \mathbb{R}}\int_0^{Ts_j}\int_0^{Tt_j}\int_0^{Ts_{j'}}\int_0^{Tt_{j'}}\notag\\
 &E_{(x_1, x_2)}\left(\left\langle N_{u_1,v_1}, \varphi\otimes \psi \right\rangle\left\langle N_{u_2,v_2}, \varphi\otimes \psi \right\rangle\right)dv_2du_2dv_1du_1\mu_1(dx_1)\mu_2(dx_2)\notag\\
 =&\sum_{j=1}^{k}d_j\sum_{j'=1}^{k}d_{j'}\frac{1}{F_T^{(1)2}}\int_{\mathbb{R}}\int_0^{Ts_j}\int_0^{Ts_{j'}}\mathcal{T}^{\alpha_1}_{u_1\wedge u_2}(\varphi\mathcal{T}^{\alpha_1}_{|u_1-u_2|}\varphi)(x_1)du_2du_1\frac{dx_1}{|x_1|^{\gamma_1}}\notag\\
 &\hspace{1.7cm}\times\frac{1}{F_T^{(2)2}}\int_{\mathbb{R}}\int_0^{Tt_j}\int_0^{Tt_{j'}}\mathcal{T}^{\alpha_2}_{v_1\wedge v_2}(\psi\mathcal{T}^{\alpha_2}_{|v_1-v_2|}\psi)(x_2)dv_2dv_1\frac{dx_2}{|x_2|^{\gamma_2}}.
\end{align}
Now, recalling (\ref{eqdefFT}) we have
\begin{align*}
&\frac{1}{\left(T^{1-(1+\gamma)/2\alpha}\right)^2}\int_{\mathbb{R}}\int_0^{Ts_1}\int_0^{Ts_2}\mathcal{T}^{\alpha}_{u_1 \wedge u_2}(\varphi \mathcal{T}^{\alpha}_{\left|u_1 - u_2\right|}\varphi)(x)du_2du_1\frac{dx}{|x|^{\gamma}}\\
=&\frac{1}{\left(T^{1-(1+\gamma)/2\alpha}\right)^2}\int_{\mathbb{R}}\int_0^{Ts_1}\int_0^{Ts_2}\int_{\mathbb{R}}p^{\alpha}_{u_1 \wedge u_2}(x-y)\varphi(y)\\
&\hspace{4.2cm}\times\int_{\mathbb{R}}p^{\alpha}_{|u_1 - u_2|}(y-z)\varphi(z)dzdydu_2du_1\frac{dx}{|x|^{\gamma}},\\
\end{align*}
substituting $u_1=Tu'_1, u_2=Tu'_2$, using the self-similarity of the $\alpha$-stable process in $\mathbb{R}$, i.e., $p^{\alpha}_t(x)=t^{-1/\alpha}p^{\alpha}_1(t^{-1/\alpha}x)$, and then substituting $x=\left(T(u'_1 \wedge u'_2)\right)^{1/\alpha}x'$, 
\begin{align}
 \label{eq51}
=&\ T^{(1+\gamma)/\alpha}\int_{\mathbb{R}}\int_0^{s_1}\int_0^{s_2}\int_{\mathbb{R}}p^{\alpha}_{T(u'_1 \wedge u'_2)}(x-y)\varphi(y)\notag\\
&\hspace{2.8cm}\times\int_{\mathbb{R}}p^{\alpha}_{T|u'_1 - u'_2|}(y-z)\varphi(z)dzdydu'_2du'_1\frac{dx}{|x|^{\gamma}}\notag\\
=&\ T^{(\gamma-1)/\alpha}\int_{\mathbb{R}}\int_0^{s_1}\int_0^{s_2}(u'_1 \wedge u'_2)^{-1/\alpha}|u'_1 - u'_2|^{-1/\alpha}\notag\\
&\hspace{2.8cm}\times\int_{\mathbb{R}}p^{\alpha}_{1}\left((T(u'_1 \wedge u'_2))^{-1/\alpha}(x-y)\right)\varphi(y)\notag\\
&\hspace{2.8cm}\times\int_{\mathbb{R}}p^{\alpha}_{1}\left((T|u'_1 - u'_2|)^{-1/\alpha}(y-z)\right)\varphi(z)dzdydu'_2du'_1\frac{dx}{|x|^{\gamma}}\notag\\
=&\int_{\mathbb{R}}\int_0^{s_1}\int_0^{s_2}(u'_1 \wedge u'_2)^{-\gamma/\alpha}|u'_1 - u'_2|^{-1/\alpha}\int_{\mathbb{R}}p^{\alpha}_{1}\left((x'-(T(u'_1 \wedge u'_2))^{-1/\alpha}y)\right)\varphi(y)\notag\\
&\hspace{2.8cm}\times\int_{\mathbb{R}}p^{\alpha}_{1}\left((T|u'_1 - u'_2|)^{-1/\alpha}(y-z)\right)\varphi(z)dzdydu'_2du'_1\frac{dx'}{|x'|^{\gamma}}.
\end{align}
Taking $T\to \infty$ in (\ref{eq51}) we obtain the limit
\begin{align}
\label{eq52}
&p^{\alpha}_{1}(0)\left(\int_{\mathbb{R}}\varphi(y)dy\right)^2\int_{\mathbb{R}}\frac{p^{\alpha}_{1}(x)}{|x|^{\gamma}}dx\int_0^{s_1}\int_0^{s_2}(u'_1 \wedge u'_2)^{-\gamma/\alpha}|u'_1 - u'_2|^{-1/\alpha}du'_2du'_1\notag\\
=&\ p^{\alpha}_{1}(0)\left(\int_{\mathbb{R}}\varphi(y)dy\right)^2\int_{\mathbb{R}}\frac{p^{\alpha}_{1}(x)}{|x|^{\gamma}}dx\notag\\
\times&\frac{1}{1-1/\alpha}\int_0^{s_1 \wedge s_2}u^{-\gamma/\alpha}\left[\left(s_1-u\right)^{1-1/\alpha}  + \left(s_2-u\right)^{1-1/\alpha}\right]du.
\end{align}

By (\ref{eq50}), (\ref{eq51}) and (\ref{eq52}),
\begin{align}
\label{e35}
&\lim_{T\to \infty}\frac{1}{F_T^2}\int_{\mathbb{R}\times\mathbb{R}}E_{(x_1, x_2)}\biggl(\sum_{j=1}^k d_j\langle L^T_{s_j, t_j};\varphi\otimes \psi\rangle\biggr)^2\mu_1(dx_1)\mu_2(dx_2)\notag\\ 
&=\frac{p^{\alpha_1}_{1}(0)}{1-1/\alpha_1}\frac{p^{\alpha_2}_{1}(0)}{1-1/\alpha_2}\biggl(\int_{\mathbb{R}}\varphi(y)dy\biggr)^2\biggl(\int_{\mathbb{R}}\psi(y)dy\biggr)^2\int_{\mathbb{R}}\frac{p^{\alpha_1}_{1}(x)}{|x|^{\gamma_1}}dx\int_{\mathbb{R}}\frac{p^{\alpha_2}_{1}(x)}{|x|^{\gamma_2}}dx\notag\\
&\times  \sum_{j,j'=1}^k d_jd_j'\int_0^{s_j \wedge s_{j'}}u^{-\gamma_1/\alpha_1}[\left(s_j-u\right)^{1-1/\alpha_1}  + \left(s_{j'}-u\right)^{1-1/\alpha_1}]du\notag\\
&\hspace{1.5cm}\times \int_0^{t_j \wedge t_{j'}}v^{-\gamma_2/\alpha_2}[\left(t_j-v\right)^{1-1/\alpha_2}  + \left(t_{j'}-v\right)^{1-1/\alpha_2}]dv\notag\\
&= D^2 \sum_{j,j'=1}^k d_jd_j' C^{(1)}(s_j, s_{j'})C^{(2)}(t_j, t_{j'}),
\end{align}
where $D$ is defined by (\ref{defD}) and 	$C^{(i)}$ is as in (\ref{equ2}) with $a_i=-\gamma_i/\alpha_i,  b_i=1-1/\alpha_i$.

Proceeding similary with the third order term we find
 \begin{align}
 \label{eq30}
 &\int_{\mathbb{R}\times \mathbb{R}}\frac{1}{F_T^3}E_{(x_1,x_2)}\biggl(\sum_{j=1}^{k}d_j\langle L_{s_j,t_j}, \varphi\otimes \psi\rangle\biggr)^3\mu_1(dx_1)\mu_2(dx_2)\notag\\
 &=\frac{1}{F_T^3}\sum_{i=1}^{k}d_i\sum_{j=1}^{k}d_j\sum_{l=1}^{k}d_l\int_{\mathbb{R}\times \mathbb{R}}E_{(x_1,x_2)}(\langle L^T_{s_i,t_i}, \varphi\otimes \psi\rangle\langle L^T_{s_j,t_j}, \varphi\otimes \psi\rangle\notag\\
 &\hspace{4cm}\times\langle L^T_{s_l,t_l}, \varphi\otimes \psi\rangle)\mu_1(dx_1)\mu_2(dx_2)\notag\\
 &=\frac{1}{F_T^3}\sum_{i=1}^{k}d_i\sum_{j=1}^{k}d_j\sum_{l=1}^{k}d_l\int_{\mathbb{R}\times \mathbb{R}}\int_0^{Ts_i}\int_0^{Tt_i}\int_0^{Ts_j}\int_0^{Tt_j}\int_0^{Ts_{l}}\int_0^{Tt_{l}}\notag\\
 &\hspace{4cm}E_{(x_1,x_2)}(\langle N_{u_i,v_i}, \varphi\otimes \psi\rangle\langle N_{u_2,v_2}, \varphi\otimes \psi\rangle\langle N_{u_3,v_3}, \varphi\otimes \psi\rangle)\notag\\
 &\hspace{4cm}dv_3du_3dv_2du_2dv_1du_1\mu_1(dx_1)\mu_2(dx_2)\notag\\
&=\sum_{i=1}^{k}d_i\sum_{j=1}^{k}d_j\sum_{l=1}^{k}d_l\frac{1}{F_T^{(1)3}}\int_{\mathbb{R}}\int_0^{Ts_i}\int_0^{Ts_j}\int_0^{Ts_{l}}\mathcal{T}^{\alpha_1}_{\tilde{u}_1}\varphi(\mathcal{T}^{\alpha_1}_{\tilde{u}_2-\tilde{u}_1}\varphi(\mathcal{T}^{\alpha_1}_{\tilde{u}_3-\tilde{u}_2}\varphi))(x_1)\notag\\
&\hspace{9.8cm}d\tilde{u}_3d\tilde{u}_2d\tilde{u}_1\frac{dx_1}{|x_1|^{\gamma_1}}\notag\\
&\times\frac{1}{F_T^{(2)3}}\int_{\mathbb{R}}\int_0^{Tt_i}\int_0^{Tt_j}\int_0^{Tt_{l}}\mathcal{T}^{\alpha_2}_{\tilde{v}_1}\psi(\mathcal{T}^{\alpha_2}_{\tilde{v}_2-\tilde{v}_1}\psi(\mathcal{T}^{\alpha_2}_{\tilde{v}_3-\tilde{v}_2}\psi))(x_2)d\tilde{v}_3d\tilde{v}_2d\tilde{v}_1\frac{dx_2}{|x_2|^{\gamma_2}},
\end{align}
$\tilde{u}_1, \tilde{u}_2, \tilde{u}_3$ denoting $u_1, u_2, u_3$ in increasing order, and similarly for $\tilde{v}_1, \tilde{v}_2, \tilde{v}_3$.

Again, recalling (\ref{eqdefFT}), substituting $\tilde{u}_i=T\tilde{u}'_i$ $i=1,2,3$, using self-similarity of the $\alpha$-stable process, and then substituting $x=\left(T\tilde{u}'_1\right)^{1/\alpha}x'$, we have
\begin{align}
\label{eq31}
&\frac{1}{\left(T^{1-(1+\gamma)/2\alpha}\right)^3}\int_{\mathbb{R}}\int_0^{Ts_1}\int_0^{Ts_2}\int_0^{Ts_{3}}\mathcal{T}^{\alpha}_{\tilde{u}_1}\varphi(\mathcal{T}^{\alpha}_{\tilde{u}_2-\tilde{u}_1}\varphi(\mathcal{T}^{\alpha}_{\tilde{u}_3-\tilde{u}_2}\varphi))(x)d\tilde{u}_3d\tilde{u}_2d\tilde{u}_1\frac{dx}{|x|^{\gamma}}\notag\\
=&\  T^{(\gamma-1)/2\alpha}\int_{\mathbb{R}}\int_0^{s_1}\int_0^{s_2}\int_0^{s_3}\int_{\mathbb{R}}\tilde{u}_1^{-\gamma/\alpha}(\tilde{u}_2 - \tilde{u}_1)^{-1/\alpha}(\tilde{u}_3 - \tilde{u}_2)^{-1/\alpha}\notag\\
&\times\int_{\mathbb{R}}p^{\alpha}_{1}(x'-(T\tilde{u}_1)^{-1/\alpha}w)\varphi(w)\int_{\mathbb{R}}p^{\alpha}_{1}((T(\tilde{u}_2 - \tilde{u}_1))^{-1/\alpha}(w-y))\varphi(y)\notag\\
&\times\int_{\mathbb{R}}p^{\alpha}_{1}((T(\tilde{u}_3 - \tilde{u}_2))^{-1/\alpha}(y-z))\varphi(z)dzdydwd\tilde{u}_3d\tilde{u}_2d\tilde{u}_1\frac{dx'}{|x'|^{\gamma}},
\end{align}
then from (\ref{eq30}) and (\ref{eq31}),
\begin{align}
\label{e33}
&\frac{1}{F_T^3}\int_{\mathbb{R}\times\mathbb{R}}E_{(x_1, x_2)}\biggl(\sum_{j=1}^k d_j\langle L^T_{s_j, t_j},\varphi\otimes \psi\rangle\biggr)^3\mu_1(dx_1)\mu_2(dx_2)\notag\\
&=\prod_{i=1}^2T^{(\gamma_i-1)/2\alpha_i}\int_{\mathbb{R}\times\mathbb{R}}A_T(x_1,x_2)\mu_1(dx_1)\mu_2(dx_2),
\end{align}
where 
\begin{align}
 \label{eq53}
&A_T(x_1,x_2)=\int_0^{s_1}\int_0^{s_2}\int_0^{s_3}\tilde{u}_1^{-\gamma_1/\alpha_1}(\tilde{u}_2 - \tilde{u}_1)^{-1/\alpha_1}(\tilde{u}_3 - \tilde{u}_2)^{-1/\alpha_1}\notag\\
&\times\int_{\mathbb{R}}p^{\alpha_1}_{1}(x_1-(T\tilde{u}_1)^{-1/\alpha_1}w)\varphi(w)\int_{\mathbb{R}}p^{\alpha_1}_{1}((T(\tilde{u}_2 - \tilde{u}_1))^{-1/\alpha_1}(w-y))\varphi(y)\notag\\
&\times\int_{\mathbb{R}}p^{\alpha_1}_{1}((T(\tilde{u}_3 - \tilde{u}_2))^{-1/\alpha_1}(y-z))\varphi(z)dzdydwdu_3du_2du_1\notag\\
&\times\int_0^{t_1}\int_0^{t_2}\int_0^{t_3}\int_{\mathbb{R}}\tilde{v}_1^{-\gamma_2/\alpha_2}(\tilde{v}_2 - \tilde{v}_1)^{-1/\alpha_2}(\tilde{v}_3 - \tilde{v}_2)^{-1/\alpha_2}\notag\\
&\times\int_{\mathbb{R}}p^{\alpha_2}_{1}(x_2-(T\tilde{v}_1)^{-1/\alpha_2}w)\psi(w)\int_{\mathbb{R}}p^{\alpha_2}_{1}((T(\tilde{v}_2 - \tilde{v}_1))^{-1/\alpha_2}(w-y))\psi(y)\notag\\
&\times\int_{\mathbb{R}}p^{\alpha_2}_{1}((T(\tilde{v}_3 - \tilde{v}_2))^{-1/\alpha_2}(y-z))\psi(z)dzdydwdv_3dv_2dv_1.
 \end{align}
 From (\ref{eq53}) we obtain
\begin{align}
 \label{eq54}
&\lim_{T\to\infty}\int_{\mathbb{R}\times\mathbb{R}}A_T(x_1,x_2)\mu_1(dx_1)\mu_2(dx_2)\notag\\
&=\int_0^{s_1}\int_0^{s_2}\int_0^{s_3}\tilde{u}_1^{-\gamma_1/\alpha_1}(\tilde{u}_2 - \tilde{u}_1)^{-1/\alpha_1}(\tilde{u}_3 - \tilde{u}_2)^{-1/\alpha_1}du_3du_2du_1\notag\\
&\times\int_0^{t_1}\int_0^{t_2}\int_0^{t_3}\tilde{v}_1^{-\gamma_2/\alpha_2}(\tilde{v}_2 - \tilde{v}_1)^{-1/\alpha_2}(\tilde{v}_3 - \tilde{v}_2)^{-1/\alpha_2}dv_3dv_2dv_1\notag\\
&\times\left(\int_{\mathbb{R}}\varphi(x)dx\right)^3\left(\int_{\mathbb{R}}\psi(x)dx\right)^3\prod_{i=1}^2(p_1^{\alpha_i}(0))^2\int_{\mathbb{R}}\frac{p_1^{\alpha_i}(x)}{|x|^{\gamma_i}}dx.
 \end{align}

Then, from (\ref{eqdelta}), (\ref{e33}) and (\ref{eq54}) we get
\begin{align}
\label{e36}
&\lim_{T\to \infty}\int_{\mathbb{R}\times\mathbb{R}} \delta_{(x_1, x_2)}^T\mu_1(dx_1)\mu_2(dx_2)=0. 
\end{align}

Finally, putting (\ref{eq4}), (\ref{e35}) and (\ref{e36}) together we obtain
\begin{align*}
&\lim_{T\to \infty}C_T(\theta) =\exp\biggl\{-\frac{\theta^2}{2}D^2\sum_{j=1}^k \sum_{j'=1}^k d_jd_{j'}C^{(1)}(s_j,s_{j'})C^{(2)}(t_j,t_{j'})\biggr\},
\end{align*}
and convergence of finite-dimensional distributions of $X_T$ to finite-dimensional distributions of weighted fractional Brownian sheet  $DW$ has been proved.
\end{proof}

\begin{theorem}
\label{teoconhbf}
 Under the hypotheses of Theorem \ref{teo1}, if $\gamma_1=\gamma_2=0$, then $X_T$ converges in law to $DW$ in the space of continuous functions $C([0,\tau]\times[0,\tau], \mathbb{R})$ for any $\tau>0$ as $T\to \infty$, where $W$ is  fractional Brownian sheet with parameters $(1-\frac{1}{2\alpha_1}, 1-\frac{1}{2\alpha_2})$, and
$$D=\int_{\mathbb{R}}\varphi(x)dx\int_{\mathbb{R}}\psi(x)dx\biggl[\prod_{i=1}^2\frac{1}{1-1/\alpha_i}p_i^{\alpha_i}(0)\biggr]^{1/2}.$$
\end{theorem}
\begin{proof}
By Theorem \ref{teo1} we have convergence of finite-dimensional distributions of $X_T	$ to those of $DW$. It remains to show that the family $\left\{X_T\right\}$ is tight. Since these processes are null on the axes, by the Bickel-Wichura theorem \cite{bickel} we only need  prove that there exist even $m\geq 2$ and positive constants  $C_m$, $\delta_1, \delta_2$ such that $m\delta_1, m\delta_2 > 1$ and
\begin{equation}
\label{eqtirantez}
\sup_T E\left(\left(\Delta_{s_1,t_1}X_T(s_2,t_2)\right)^m \right)\leq C_m(s_2 - s_1)^{m\delta_1}(t_2 - t_1)^{m\delta_2}, \ \  \text{for all} \ s_1<s_2, t_1<t_2.
\end{equation} 

From (\ref{eqdefXT}), 
\begin{align*}
\Delta_{s_1,t_1}X_T(s_2,t_2)=& \frac{1}{F_T}\int_{Ts_1}^{Ts_2} \int_{Tt_1}^{Tt_2}\left(\left\langle N_{u,v},\varphi\otimes \psi \right\rangle - E(\left\langle N_{u,v},\varphi\otimes \psi \right\rangle)\right)dudv,
\end{align*}
then, by (\ref{defcov}), 
\begin{align}
\label{eq310}
&E(\left(\Delta_{s_1,t_1}X_T(s_2,t_2)\right)^2) \notag \\
&={\frac{1}{F_T^2} \int_{Ts_1}^{Ts_2} \int_{Tt_1}^{Tt_2}\int_{Ts_1}^{Ts_2} \int_{Tt_1}^{Tt_2}}Cov\left(\left\langle N_{u_1,v_1},\varphi\otimes \psi \right\rangle,\left\langle N_{u_2,v_2},\varphi\otimes \psi \right\rangle\right)dv_2du_2dv_1du_1\notag\\
&=\frac{1}{F_T^2} \int_{Ts_1}^{Ts_2} \int_{Tt_1}^{Tt_2}\int_{Ts_1}^{Ts_2} \int_{Tt_1}^{Tt_2}\int_{\mathbb{R}}\mathcal{T}^{\alpha_1}_{u_1 \wedge u_2}(\varphi \mathcal{T}^{\alpha_1}_{\left|u_1 - u_2\right|}\varphi)(x_1)dx_1\notag\\
&\hspace{4cm}\times \int_{\mathbb{R}}\mathcal{T}^{\alpha_2}_{v_1 \wedge v_2}(\varphi \mathcal{T}^{\alpha_2}_{\left|v_1 - v_2\right|}\varphi)(x_2)dx_2dv_2du_2dv_1du_1\notag\\
&=\frac{1}{T^{2-(1+\gamma_1)/\alpha_1}}\int_{Ts_1}^{Ts_2}\int_{Ts_1}^{Ts_2}\int_{\mathbb{R}}\mathcal{T}^{\alpha_1}_{u_1 \wedge u_2}(\varphi \mathcal{T}^{\alpha_1}_{\left|u_1 - u_2\right|}\varphi)(x_1)dx_1du_2du_1\notag\\
&\times  \frac{1}{T^{2-(1+\gamma_2)/\alpha_2}}\int_{Tt_1}^{Tt_2}\int_{Tt_1}^{Tt_2}\int_{\mathbb{R}}\mathcal{T}^{\alpha_2}_{v_1 \wedge v_2}(\varphi \mathcal{T}^{\alpha_2}_{\left|v_1 - v_2\right|}\varphi)(x_2)dx_2dv_2dv_1\notag\\
&=E(\langle X_T^{(1)}(s_2)-X_T^{(1)}(s_1), \varphi\rangle^2)E(\langle X_T^{(2)}(t_2)-X_T^{(2)}(t_1), \psi\rangle^2),
\end{align}
where $X_T^{(i)}, i=1,2,$ are occupation time fluctuation processes of independent systems of particles moving in $\mathbb{R}$ according to symmetric $\alpha_i$-stable processes with initial configurations given by  Poisson random measures on $\mathbb{R}$ with intensities $\mu_i$, i.e.,
$$X_T^{(1)}(t)=\frac{1}{T^{1-1/2\alpha_1}}\int_{0}^{Tt}(\langle N_u^{(1)}, \varphi\rangle-E(\langle N_u^{(1)}, \varphi\rangle))du, \ \ \ \  N_u^{(1)}= \sum_{x\in \text{Pois}(\mu_1)} \delta_{\xi_u^{x}},$$
and
$$X_T^{(2)}(t)=\frac{1}{T^{1-1/2\alpha_2}}\int_{0}^{Tt}(\langle N_u^{(2)}, \varphi\rangle-E(\langle N_u^{(2)}, \varphi\rangle))du, \ \ \ \  N_u^{(2)}= \sum_{x\in \text{Pois}(\mu_2)} \delta_{\zeta_u^{x}}.$$
In \cite{bgt2006} such a one-time system is studied and it is shown that 
\begin{align}
\label{eq311}
E(\langle X_T^{(i)}(t)-X_T^{(i)}(s), \varphi\rangle^2)\leq C|t-s|^h,
\end{align}
where $C$ is a positive constant (not depending on $T$) and  $h=2-\frac{1}{\alpha_i}>1$. From (\ref{eq310}) and (\ref{eq311}) we obtain (\ref{eqtirantez}).  
\end{proof}

\begin{remark}
\begin{enumerate}
\item Theorem \ref{teoconhbf} gives a functional approximation  of fractional Brownian sheet with parameters $(h_1, h_2)\in (1/2, 3/4]^2$, taking   $h_i=1-\frac{1}{2\alpha_i}$, $i=1,2.$ 
	
	\item Proving tightness with $\gamma_1\neq 0$ or $\gamma_2\neq 0$ is considerably more difficult because it requires computing moments of arbitrarily high order (see \cite{bgt2007} for the one time case), and this involves moments of arbitrarily high order of the Poisson random measure $\text{Pois}(\mu_1\otimes\mu_2)$, which are cumbersome.
	
	\item In Theorem \ref{teo1} we may also consider the measures $\mu_i$ of the form (\ref{eqmu}) with $\gamma_i<0$, assuming that $|\gamma_i|<\alpha_i$ if $\alpha_i<2$ (which implies that the mean is finite), and the result in the theorem holds. 
	
	\item The role of the functions $\varphi, \psi$ is only subsidiary since they are fixed, and in the occupation time fluctuation limit they appear only in the constant $D$ given by (\ref{defD}). If $\varphi, \psi$ are taken as variables in the space $\mathcal{S}(\mathbb{R})$ of smooth rapidly decreasing functions, then in principle it is possible to prove convergence of the occupation time fluctuations as $(\mathcal{S}'(\mathbb{R}))^2$-valued processes, where $\mathcal{S}'(\mathbb{R})$ is the space of tempered distributions (topological dual of $\mathcal{S}(\mathbb{R})$), the limit being the space-time random field
	$$\left(Z_{s,t}\right)_{s,t \geq 0} = K(\lambda \otimes \lambda)\left(W_{s,t}\right)_{s,t \geq 0},$$
	where $W$ is the weighted fractional Brownian sheet in Theorem \ref{teo1},
	$$K=\left(\prod_{i=1}^2\frac{1}{1-1/\alpha_i}p_1^{\alpha_i}(0)\int_{\mathbb{R}}\frac{p_1^{\alpha_i}(x)}{|x|^{\gamma_i}}dx\right)^{1/2},$$
	and $\lambda$ is the Lebesgue measure on $\mathbb{R}$. (See \cite{bgt2007} for such a setup for a one-time particle system.)
\end{enumerate}
\end{remark}

\bibliographystyle{amsplain}

\end{document}